\documentclass[12pt]{article}
\usepackage{fullpage,amsthm,amssymb,amsmath}
\usepackage{graphicx,algorithmic,algorithm,amsfonts}
\usepackage{enumerate,color}
\usepackage{tikz}

\newtheorem{thm}{Theorem}[section]
\newtheorem*{thm:bal2}{Theorem~\ref{thmm:bal2}}
\newtheorem*{thm:unbal2}{Theorem~\ref{thmm:unbal2}}
\newtheorem*{thm:unbal3}{Theorem~\ref{thmm:unbal3}}
\newtheorem{lem}[thm]{Lemma}
\newtheorem{claim}[thm]{Claim}
\newtheorem{problem}[thm]{Problem}

\newcommand\thmbaltwo{
A set $S$ of cycles is an inclusion-wise minimal cycle obstruction set of balanced $2$-partitionable planar graphs if and only if $S=\{C_4\}$ or $S$ is the set of all odd cycles.
}
\newcommand\thmunbaltwo{
A set $S$ of cycles is an inclusion-wise minimal cycle obstruction set of unbalanced $2$-partitionable planar graphs if and only if $S=\{C_3,C_4,C_6\}$ or $S$ is the set of all odd cycles.
}
\newcommand\thmunbalthree{
A set $S$ of cycles is an inclusion-wise minimal cycle obstruction set of unbalanced $3$-partitionable planar graphs if and only if $S=\{C_3\}$ or $S=\{C_4\}$.
}

\begin{document}

\title{Characterization of cycle obstruction sets for improper coloring planar graphs}

\author{
Ilkyoo Choi\thanks{
Supported by the National Research Foundation of Korea (NRF) grant funded by the Korea government (MSIP) (NRF-2015R1C1A1A02036398).
This work was supported by Hankuk University of Foreign Studies Research Fund. 
Department of Mathematics, Hankuk University of Foreign Studies, Yongin-si, Gyeonggi-do, Republic of Korea
\texttt{ilkyoo@hufs.ac.kr}
}
\and
Chun-Hung Liu\thanks{
Partially supported by the National Science Foundation under Grant No. DMS-1664593.
Department of Mathematics, Princeton University, Princeton, NJ, USA.
\texttt{chliu@math.princeton.edu}
}
\and
Sang-il Oum\thanks{
This work was supported by the National Research Foundation of Korea (NRF) grant funded by the Korea government (MSIT) (No. NRF-2017R1A2B4005020).
Department of Mathematical Sciences, KAIST, Daejeon, Republic of Korea.
School of Mathematics, KIAS, Seoul, Republic of Korea. 
\texttt{sangil@kaist.edu}
}
}

\date\today

\maketitle

\begin{abstract}
For nonnegative integers $k, d_1, \ldots, d_k$, a graph is $(d_1, \ldots, d_k)$-colorable if its vertex set can be partitioned into $k$ parts so that the $i$th part induces a graph with maximum degree at most $d_i$ for all $i\in\{1, \ldots, k\}$. 
A class $\mathcal C$ of graphs is {\it balanced $k$-partitionable} and {\it unbalanced $k$-partitionable} if there exists a nonnegative integer $D$ such that all graphs in $\mathcal C$ are $(D, \ldots, D)$-colorable and $(0, \ldots, 0, D)$-colorable, respectively, where the tuple has length $k$. 

A set $X$ of cycles is a {\it cycle obstruction set} of a class $\mathcal C$ of planar graphs if every planar graph containing none of the cycles in $X$ as a subgraph belongs to $\mathcal C$.
This paper characterizes all cycle obstruction sets of planar graphs to be balanced $k$-partitionable and unbalanced $k$-partitionable for all $k$; 
namely, we identify all inclusion-wise minimal cycle obstruction sets for all $k$. 
\end{abstract}

\section{Introduction}

All graphs in this paper are finite and simple, which means no loops and no parallel edges. Let $C_k$ denote a $k$-cycle, which is a cycle of length $k$.
A set $X$ of cycles is a {\it cycle obstruction set} of a class $\mathcal C$ of planar graphs if every planar graph containing none of the cycles in $X$ as a subgraph belongs to $\mathcal C$.

A graph is {\it $k$-colorable} if its vertex set can be partitioned into $k$ color classes so that each color class is an independent set.
The celebrated Four Color Theorem~\cite{1977ApHa,1977ApHaKo} (later reproved in~\cite{1997RoSaSeTh}) states that every planar graph is $4$-colorable. 
Since there are planar graphs that are not $3$-colorable, finding sufficient conditions for a planar graph to be $3$-colorable has been an active area of research; many of these conditions can be translated into the language of obstruction sets. 
Perhaps the most well-known result is the following theorem, known as Gr\"otzsch's Theorem~\cite{1959Gr}:

\begin{thm}[Gr\"otzsch~\cite{1959Gr}]
Planar graphs with no $3$-cycles are $3$-colorable.
\end{thm}

In the language of obstruction sets, Gr\"otzsch's Theorem states that $\{C_3\}$ is an obstruction set of $3$-colorable planar graphs. 
There is also a vast literature regarding forbidding various cycle lengths to guarantee a planar graph to be $3$-colorable;
see Table~\ref{tab:history} for a summary of some of these results. 


\begin{table}[h]\label{tab:history}
\begin{center}
\begin{tabular}{c||c||c|c|c|c|c|c|c||l}
year & reference  & $ 3$  &  $ 4$  &  $ 5$  &  $ 6$  &  $ 7$  &  $ 8$  &  $ 9$ & authors   \\
\hline \hline
1959 &\cite{1959Gr} & $\times$ &    &    &    &    &      &  &  Gr\"otzsch \\\hline
2005 & \cite{2005ZhWu} &    & $\times$ & $\times$ & $\times$ &    & & $\times$  & Zhang--Wu \\\hline
2006 & \cite{2006Xu} && $\times$ & $\times$ &    & $\times$ &    &     & Xu \\\hline
2010 & \cite{2010WaLuCh} && $\times$ & $\times$ &    &    & $\times$ & $\times$ &  Wang--Lu--Chen \\\hline
2007 & \cite{2007ChRaWa} &    & $\times$ &    & $\times$ & $\times$ &    & $\times$  &Chen--Raspaud--Wang \\\hline
2007 & \cite{2007WaCh} & &$\times$ &    & $\times$ &    & $\times$ &    &  Wang--Chen  \\\hline
2011 & \cite{2011WaWuSh} &    & $\times$ &    &    & $\times$ & $\times$ & $\times$ & Wang--Wu--Shen \\
\end{tabular}
\end{center}
\caption{Forbidding various cycle lengths to guarantee 3-colorability of planar graphs}
\end{table}

Each result in the aforementioned theorem reveals a new obstruction set of $3$-colorable planar graphs. 
The interest in forbidding various cycle lengths stems from Steinberg's Conjecture~\cite{1993St}, which states that planar graphs with neither $4$-cycles nor $5$-cycles are $3$-colorable. 
There was almost no progress after the conjecture was first proposed in 1976, but many partial results were produced after 1991, which is when Erd\H{o}s~\cite{1993St} proposed the following approach towards Steinberg's Conjecture: find the minimum $k$ such that planar graphs with no cycle lengths in $\{4, \ldots, k\}$ are $3$-colorable. 
After 40 years of effort by the coloring community to try to prove Steinberg's Conjecture, only recently it was disproved via a clever construction by Cohen-Addad et al.~\cite{2017CoHeKrLiSa}.
Yet, the question of whether planar graphs with no cycle lengths in $\{4, 5, 6\}$ are $3$-colorable or not remains open.

Recently, the following relaxation of proper coloring, also known as {\it improper coloring}, has attracted much attention: for nonnegative integers $k, d_1, \ldots, d_k$, a graph is {\it $(d_1, \ldots, d_k)$-colorable} if its vertex set can be partitioned into $k$ color classes $V_1, \ldots, V_k$ so that $V_i$ induces a graph with maximum degree at most $d_i$ for all $i\in\{1, \ldots, k\}$. 
This relaxation allows some prescribed defects in each color class, where defects are measured in terms of the maximum degree of the graph induced by the vertices of a color class. 
We say a class $\mathcal C$ of graphs is {\it balanced $k$-partitionable} and {\it unbalanced $k$-partitionable} if there exists a nonnegative integer $D$ such that all graphs in $\mathcal C$ are $(D, \ldots, D)$-colorable and $(0, \ldots, 0, D)$-colorable, respectively, where the tuple has length $k$. 

There is a vast literature in improper coloring planar graphs. 
By the Four Color Theorem, planar graphs are $4$-colorable, which is equivalent to $(0, 0, 0, 0)$-colorable, and Cowen et al.\ \cite{1986CoCoWo} proved that planar graphs are $(2, 2, 2)$-colorable.
This is best possible in the sense that for any given nonnegative integers $d_1$ and $d_2$, there exists a planar graph that is not $(1, d_1, d_2)$-colorable; for one such construction see~\cite{unpub_ChEs}. 
Therefore, the question of partitioning planar graphs with no extra conditions into at least three subgraphs of bounded maximum degrees is completely solved. 

It is often useful to consider girth conditions along with planarity to obtain positive results. 
Regarding partitioning planar graphs into two parts, for any given nonnegative integers $d_1$ and $d_2$, a planar graph with girth $4$ that is not $(d_1, d_2)$-colorable is constructed in~\cite{2015MoOc}. 
Yet, Choi et al.~\cite{2016ChChJeSu}, Borodin and Kostochka~\cite{2014BoKo}, Choi and Raspaud~\cite{2015ChRa}, and \v Skrekovski~\cite{2000Sk} proved that planar graphs with girth at least $5$ are $(1, 10)$-, $(2, 6)$-, $(3, 5)$-, and $(4, 4)$-colorable, respectively. 
Also, given a nonnegative integer $d$, a planar graph with girth $6$ that is not $(0, d)$-colorable is constructed in~\cite{2010BoIvMoOcRa}.
On the other hand, it is known that every planar graph with girth at least $7$ is $(0, 4)$-colorable~\cite{2014BoKo}. 
For other papers regarding improper coloring sparse (not necessarily planar) graphs, see~\cite{2013BoKoYa,2013EsMoOcPi,2006HaSe,2014KiKoZh,2016KiKoZh}.

The previous paragraph concerns girth conditions enforced on planar graphs to obtain positive results. 
Instead of forbidding all short cycles, we are interested in finding the minimal sets of obstacles in partitioning planar graphs into parts with bounded maximum degrees. 
We succeed in identifying which cycle lengths are essential obstructions when it comes to partitioning planar graphs in a balanced and unbalanced way.
In other words, this paper characterizes all cycle obstruction sets of balanced $k$-partitionable and unbalanced $k$-partitionable planar graphs for all $k$; namely, we identify all the inclusion-wise minimal cycle obstruction sets. 

By the Four Color Theorem, the empty set is the (only) inclusion-wise minimal cycle obstruction set of both balanced $k$-partitionable and unbalanced $k$-partitionable planar graphs when $k\geq 4$.
The empty set is also the (only) inclusion-wise minimal cycle obstruction set of balanced $3$-partitionable planar graphs, since Cowen et al.\ \cite{1986CoCoWo} proved that planar graphs are $(2, 2, 2)$-colorable.
For the remaining cases, we characterize the inclusion-wise minimal obstruction sets, and for each case there are exactly two. 
Our main results are the following three theorems:

\begin{thm:bal2}
\thmbaltwo
\end{thm:bal2}

\begin{thm:unbal2}
\thmunbaltwo
\end{thm:unbal2}

\begin{thm:unbal3}
\thmunbalthree
\end{thm:unbal3}

Theorem~\ref{thmm:bal2} and Theorem~\ref{thmm:unbal2} state that for planar graphs to be balanced $2$-partitionable and unbalanced $2$-partitionable, respectively, there is only one inclusion-wise minimal cycle obstruction set other than the set of all odd cycles.
Since forbidding all odd cycles makes the graph bipartite, and thus $2$-colorable, which is equivalent to $(0, 0)$-colorable, the minimal cycle obstructions for planar graphs to be balanced $2$-partitionable and unbalanced $2$-partitionable is a $4$-cycle and all of $3$-, $4$-, $6$-cycles, respectively.
Note that previous results by \v Skrekovski~\cite{2000Sk} and Borodin and Kostochka~\cite{2014BoKo} imply that planar graphs are balanced $2$-partitionable and unbalanced $2$-partitionable when the forbidden cycle lengths are $3, 4$ and $3, 4, 5, 6$, respectively. 

Theorem~\ref{thmm:unbal3} states that other than a $3$-cycle, there is only one other inclusion-wise minimal cycle obstruction set of unbalanced $3$-partitionable planar graphs. 
Since Gr\"otzsch's Theorem says that forbidding a $3$-cycle in planar graphs guarantees that it is $3$-colorable, which is equivalent to $(0, 0, 0)$-colorable, the minimal cycle obstruction for non-3-colorable planar graphs to be unbalanced $3$-partitionable is a $4$-cycle. 

Note that for both balanced $1$-partitioning and unbalanced $1$-partitioning, cycle obstruction sets simply do not exist because of planar graphs with arbitrarily large maximum degree. 
See Table~\ref{tab:results} for a complete list of cycle obstruction sets of both balanced $k$-partitionable and unbalanced $k$-partitionable planar graphs. 

\begin{table}[h]
\begin{center}
\renewcommand{\arraystretch}{1.5}
\begin{tabular}{c||c|c}
$ k$ & balanced & unbalanced \\ \hline \hline 
${4^+}$-{partitionable} & {$\emptyset$} & {$\emptyset$} \\ \hline
$ 3$-{partitionable}     & {$\emptyset$} & $\{ C_3\}, \{ C_4\}$\\ \hline
$ 2$-{partitionable}     
& $\{C_{2i+1}:i\geq 1\}$, $\{ C_4\}$ 
& $\{C_{2i+1}:i\geq 1\}$, $\{ C_3,  C_4,  C_6\}$\\
\hline
$ 1$-{partitionable}     & does not exist! & does not exist!\\
\end{tabular}
\end{center}
\caption{Characterization of inclusion-wise minimal cycle obstruction sets}\label{tab:results}
\end{table}

In Section~\ref{sec:bal2}, Section~\ref{sec:unbal2}, and Section~\ref{sec:unbal3}, we prove Theorem~\ref{thmm:bal2}, Theorem~\ref{thmm:unbal2}, and Theorem~\ref{thmm:unbal3}, respectively. 
The constants in all of our main results are probably improvable with some effort. 
Yet, we focused on simplifying the proofs and using the minimum number of reducible configurations and basic discharging rules in order to improve the readability of the paper. 
We end this section by posing three questions and some definitions that will be used in the next sections. 

\begin{problem}
What is the minimum $D$ such that every planar graph with no $4$-cycles is $(D, D)$-colorable?
\end{problem}

\begin{problem}
What is the minimum $D$ such that every planar graph with no $3$-, $4$-, $6$-cycles is $(0, D)$-colorable?
\end{problem}

\begin{problem}
What is the minimum $D$ such that every planar graph with no $4$-cycles is $(0, 0, D)$-colorable?
\end{problem}


The {\it degree} of a vertex $v$, denoted by $d(v)$, is the number of edges incident with it.
A {\it $k$-vertex}, {\it $k^+$-vertex}, and {\it $k^-$-vertex} is a vertex of degree exactly $k$, at least $k$, and at most $k$, respectively. 
Given any embedding of a connected planar graph $G$ on at least two vertices on the plane, for every face $f$, we say that a boundary walk $W_f$ of $f$ is {\it canonical} if it traces the edges incident with $f$ according to one of the two obvious cyclic orderings of those edges.
The {\it degree} of a face $f$, denoted by $d(f)$, is the length of $W_f$; note that cut edges are counted twice.
A {\it $k$-face}, {\it $k^+$-face}, and {\it $k^-$-face} is a face of degree exactly $k$, at least $k$, and at most $k$, respectively. 
For each face $f$ and each vertex $v$ of $G$, we define $k_{f,v}$ to be the number of triples $(e,v,e')$ such that $e,e' \in E(G)$ and $eve'$ is a subwalk of $W_f$.
It is well-known that the degree of $f$ and $k_{f,v}$ is independent of the choice of $W_f$.
Clearly, the degree of $f$ equals $\sum_{v \in V(G)} k_{f,v}$.


\section{Balanced $2$-partitions}\label{sec:bal2}

In this section, we prove the following theorem:

\begin{thm}\label{thmm:bal2}
\thmbaltwo
\end{thm}

We will first show a necessary condition for cycle obstruction sets, and then show that it is sufficient afterwards. 

\begin{lem}\label{lem:bal2:nec}
If a set $S$ of cycles is an obstruction set of balanced $2$-partitionable planar graphs, then either $C_4 \in S$ or $S$ contains all odd cycles.
\end{lem}
\begin{proof}
Given a nonnegative integer $D$ and two vertices $x$ and $y$, let $H_2(D; x, y)$ be the graph consisting of $2D+1$ internally disjoint $x, y$-paths of length $2$. 
For a positive integer $l$ and a vertex $v_1$, let $H_1(D, l; v_1)$ be the graph obtained from a cycle with vertices $v_1, \ldots, v_{l+1}$ and replacing each edge $v_iv_{i+1}$ with a copy of $H_2(D; v_i, v_{i+1})$ where $i\in\{1, \ldots, l\}$. 
Finally, let $H(D, l)$ be the graph obtained from $D+1$ pairwise disjoint copies of $H_1(D, l; v^j_1)$ and identifying all of $v^j_1$ for $j\in\{1, \ldots, D+1\}$. 

Now in any $(D, D)$-coloring of $H_2(D; x, y)$, it is easy to see that $x$ and $y$ must receive the same color. 
This implies that the cut-vertex of $H(D, l)$ has $D+1$ neighbors of the same color, which shows that $H(D, l)$ is not $(D, D)$-colorable. 
It is not hard to see that the cycles in $H(D, l)$ have length either $4$ or $2l+1$. 
Therefore the obstruction set of balanced $2$-partitionable planar graphs contains either $C_4$ or all odd cycles. 
See Figure~\ref{fig:bal2-tight} for an illustration of $H_2(D; x, y)$ and $H(D, 2)$.
\end{proof}

\begin{figure}[ht]
	\begin{center}
  \includegraphics[scale=1]{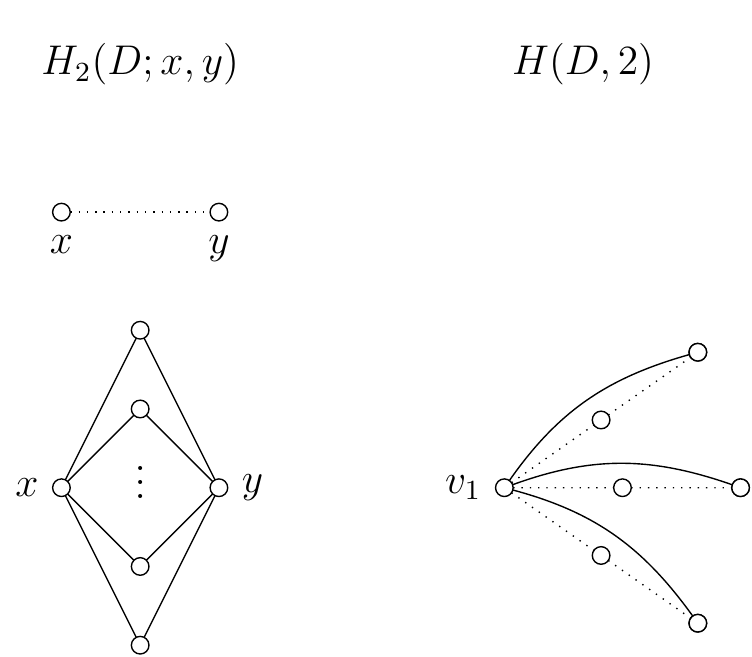}
  \caption{Graphs that are not $(D, D)$-colorable}
  \label{fig:bal2-tight}
	\end{center}
\end{figure}

If a planar graph does not contain any odd cycles, then it is bipartite, and thus it is $(0, 0)$-colorable, and hence it is balanced $2$-partitionable. 
The remaining of this section proves that planar graphs with no $4$-cycles are balanced $2$-partitionable. 
Note that Lemma~\ref{lem:bal2:nec} and Theorem~\ref{thm:bal2} imply Theorem~\ref{thmm:bal2}.

\begin{thm}\label{thm:bal2}
A planar graph with no $4$-cycles is $(5, 5)$-colorable. 
\end{thm}

In the rest of this section, let $G$ be a counterexample to Theorem~\ref{thm:bal2} with the minimum number of $3^+$-vertices, and subject to that choose one with the minimum number of edges. 
Also, fix a plane embedding of $G$. 
It is easy to see that $G$ is connected and has no $1$-vertices.
From now on, given a (partially) $(5, 5)$-colored graph, we will let $a$ and $b$ be the two colors, and we say a vertex with a color is {\it saturated} if it already has five neighbors of the same color.

\subsection{Structural lemmas}

\begin{lem}\label{lem:bal2:edge}
Every edge $xy$ of $G$ has an endpoint with degree at least $7$. 
\end{lem}
\begin{proof}
Suppose to the contrary that $x$ and $y$ are both $6^-$-vertices. 
Since $G\setminus xy$ is a graph with fewer edges than $G$ and the number of $3^+$-vertices did not increase, there is a $(5, 5)$-coloring $\varphi:V(G)\rightarrow\{a, b\}$ of $G\setminus xy$. 
If $\varphi$ is not a $(5, 5)$-coloring of $G$, then $\varphi(x)=\varphi(y)$, and either $x$ or $y$ is saturated in $G\setminus xy$. 
For each saturated vertex $z$ in $\{x, y\}$, we may recolor it with the color in $\{a, b\}\setminus\{\varphi(z)\}$ since all of its neighbors have color $\varphi(z)$ in $G\setminus xy$.
We end up with a $(5, 5)$-coloring of $G$, which is a contradiction. 
\end{proof}

\begin{lem}\label{lem:bal2:3-vx}
There are no $3$-vertices in $G$. 
\end{lem}
\begin{proof}
Suppose to the contrary that $v$ is a $3$-vertex of $G$ with neighbors $v_1, v_2, v_3$. 
By Lemma~\ref{lem:bal2:edge}, we know that $v_1, v_2, v_3$ are $7^+$-vertices. 
Obtain a graph $H$ from $G-v$ by adding paths $v_1u_1v_2, v_2u_2v_3, v_3u_3v_1$, 
 where $u_1,u_2,u_3$ are three distinct vertices not in $G$.
Note that $H$ is planar and has no $4$-cycles since the pairwise distance between $v_1, v_2, v_3$ did not change. 
See Figure~\ref{fig:bal2-3-vx} for an illustration.
Since $H$ has fewer $3^+$-vertices than $G$, there is a $(5, 5)$-coloring $\varphi:V(H)\rightarrow\{a, b\}$ of $H$. 

If $\varphi(v_1)=\varphi(v_2)=\varphi(v_3)$, then we may extend $\varphi$ to $G$ by using the color in $\{a, b\}\setminus\{\varphi(v_1)\}$ on $v$. 
Otherwise, without loss of generality we may assume $\varphi(v_1)=a$ and $\varphi(v_2)=\varphi(v_3)=b$. 
If $a\in\{\varphi(u_1), \varphi(u_3)\}$, then we may extend $\varphi$ to $G$ by using $a$ on $v$.
Otherwise, $\varphi(u_1)=\varphi(u_3)=b$, so we may extend $\varphi$ to $G$ by using $b$ on $v$.
In all cases we end up with a $(5, 5)$-coloring of $G$, which is a contradiction. 
\end{proof}

\begin{figure}[ht]
	\begin{center}
  \includegraphics[scale=1]{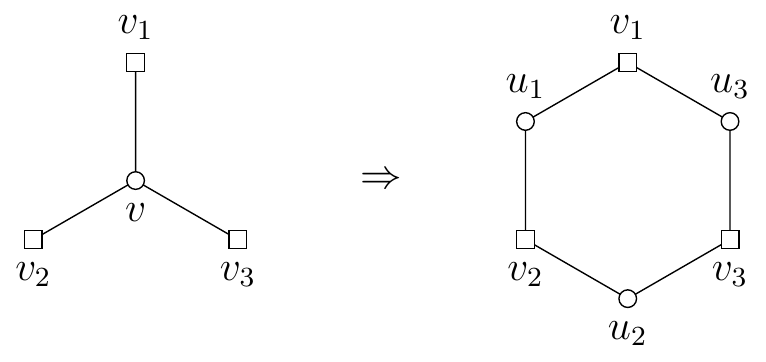}
  \caption{Obtaining $H$ from $G$ in Lemma~\ref{lem:bal2:3-vx}}
  \label{fig:bal2-3-vx}
	\end{center}
\end{figure}

A $3$-face is {\it terrible} if it is incident with a $2$-vertex. 

\begin{lem}\label{lem:bal2:terrible}
A $7^+$-vertex $v$ is incident with at most $\min\{\lfloor{d(v)\over 2}\rfloor, d(v)-6\}$ terrible $3$-faces. 
\end{lem}
\begin{proof}
Since $G$ has no $4$-cycles, two $3$-faces cannot share an edge, and thus $v$ is incident with at most $\lfloor{d(v)\over 2}\rfloor$ terrible 3-faces.
Since $\lfloor{d(v)\over 2}\rfloor\leq d(v)-6$ when $d(v)\geq 11$, we may assume $d(v)\leq 10$. 

Suppose to the contrary that $v$ is incident with $t$ terrible 3-faces, where $t \geq d(v)-5$. 
Let $w$ be a $2$-vertex of a terrible 3-face $wvu$; note that $u$ is also a $7^+$-vertex by Lemma~\ref{lem:bal2:edge}. 
Since $G-w$ is a graph with fewer edges than $G$ and the number of $3^+$-vertices did not increase, there is a $(5, 5)$-coloring $\varphi:V(G)\setminus\{w\}\rightarrow\{a, b\}$ of $G-w$. 
If $\varphi(u)=\varphi(v)$, then we may extend $\varphi$ to $G$ by using the color in $\{a, b\}\setminus\{\varphi(u)\}$ on $w$. 
Thus, we may assume $\varphi(u)=a$ and $\varphi(v)=b$. 
Since using the color $b$ on $w$ should not extend $\varphi$ to $G$, we know that $v$ must be saturated by $\varphi$. 


There are $d(v)-2t$ neighbors of $v$ in $G-w$ that are not in terrible 3-faces incident with $v$. 
Since $v$ has five neighbors with the color $b$, at least $5-(d(v)-2t)=5+2t-d(v)$ neighbors of $v$ in $G-w$ with the color $b$ are incident with a terrible 3-face incident with $v$. 
Since neither $w$ nor $u$ is colored with $b$, there are $t-1$ terrible 3-faces incident with $v$ that might have a vertex colored with $b$. 
Since $t\geq d(v)-5$ implies $5+2t-d(v)>t-1$, there exists a terrible 3-face $xyv$ where $x$ is a 2-vertex, other than $wuv$ with $\varphi(x)=\varphi(y)=b$. 
Now, we can extend $\varphi$ to $G$ by coloring $w$ with $b$ and recoloring $x$ with $a$, which contradicts the assumption that $G$ has no $(5, 5)$-coloring. 
\end{proof}

\subsection{Discharging}


We now define the initial charge at each vertex and each face.
For every $v\in V(G)$, let $\mu(v)=2d(v)-6$ and for every face $f\in F(G)$, let $\mu(f)=d(f)-6$. 
The total initial charge is negative since
\begin{align*}
\sum_{z\in V(G)\cup F(G)} \mu(z)
	=\sum_{v\in V(G)} (2d(v)-6)+\sum_{f\in F(G)} (d(f)-6) 
	=-6|V(G)|+6|E(G)|-6|F(G)|
	=-12	<0.\end{align*}
The last equality holds by Euler's formula.
Recall that a $3$-face is terrible if it is incident with a $2$-vertex. 

Here are the discharging rules:

\begin{enumerate}[(R1)]

\item Each $7^+$-vertex sends charge $1$ to each adjacent $2$-vertex. 
\item Each $4$-, $5$-, $6$-vertex sends charge $1$ to each incident $3$-face. 
\item Let $v$ be a $7^+$-vertex.

\begin{enumerate}[(R3A)]

\item $v$ sends charge $3\over 2$ to each incident terrible 3-face.
\item $v$ sends charge $1$ to each incident $3$-face that is not terrible.
\item $v$ sends charge $1\over 2$ to each $5$-face $f$ that is incident with $v$ and incident with a neighbor of $v$ with degree at least $7$.

\end{enumerate}

\end{enumerate}
See Figure~\ref{fig:bal2-rules} for an illustration of the discharging rules. 

\begin{figure}[h]
	\begin{center}
		\includegraphics{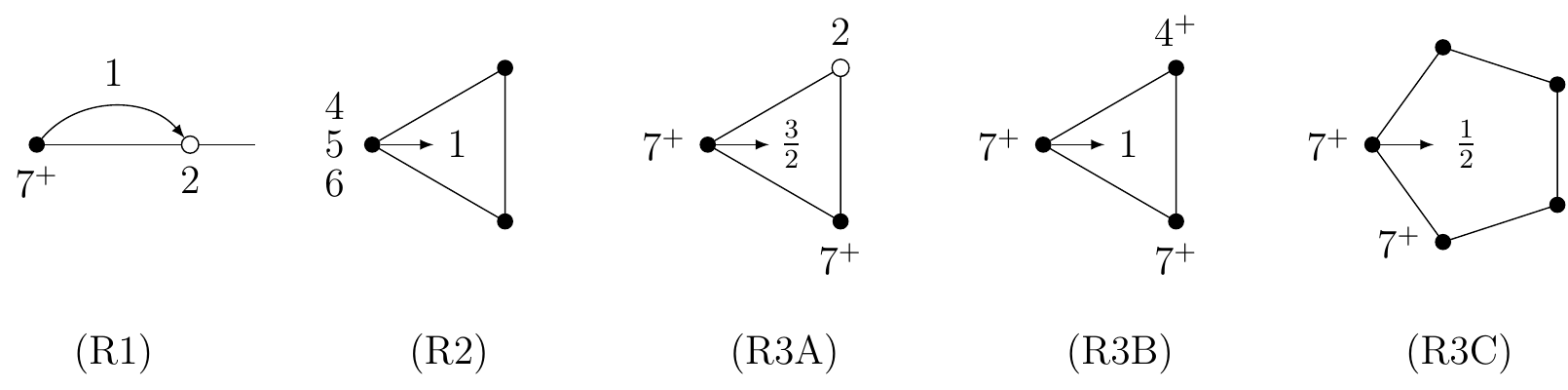}
	\end{center}
  \caption{Discharging rules}
  \label{fig:bal2-rules}
\end{figure}

We denote the final charge of $z$ by $\mu^*(z)$ for each $z \in V(G) \cup F(G)$.
The rest of this section will prove that $\mu^*(z)$ is nonnegative for each $z\in V(G)\cup F(G)$. 

\begin{claim}
Every face has nonnegative final charge. 
\end{claim}
\begin{proof}
Let $f$ be a face.
It only receives charge and does not give out any charge. 
Thus if $f$ is a $6^+$-face, then $f$ has nonnegative final charge since $\mu^*(f)=\mu(f)=d(f)-6\geq 0$. 
By Lemma~\ref{lem:bal2:edge}, every $5$-face $f$ is incident with at least three $7^+$-vertices, and at least two of these are adjacent to each other. 
Thus, by rule (R3C), $f$ receives charge $1\over 2$ at least twice. 
Thus, $\mu^*(f)\geq 5-6+2\cdot{1\over 2}=0$. 
Note that $f$ cannot be a $4$-face since $G$ has no $4$-cycles. 

Now assume $f$ is a $3$-face. 
By Lemma~\ref{lem:bal2:edge}, $f$ is incident with at least two $7^+$-vertices, which must be pairwise adjacent to each other.
If $f$ is incident with two $7^+$-vertices and the third vertex is a $4^+$-vertex, then $f$ is not a terrible face.
Now $f$ receives either charge $1$ twice by rule (R3B) and charge $1$ once by rule (R2) or charge $1$ three times by rule (R3B). 
In either case, $\mu^*(f)=3-6+3\cdot{1}=0$.
Note that there are no $3$-vertices by Lemma~\ref{lem:bal2:3-vx}.
If $f$ is incident with exactly two $7^+$-vertices, then the third vertex is a $2$-vertex, and $f$ is a terrible 3-face. 
Thus it receives charge ${3\over 2}$ twice by rule (R3A).
Thus, $\mu^*(f)=3-6+2\cdot{3\over 2}=0$.
\end{proof}

\begin{claim}
Each vertex has nonnegative final charge.
\end{claim}
\begin{proof}
Each neighbor of a $2$-vertex $v$ must be a $7^+$-vertex by Lemma~\ref{lem:bal2:edge}. 
Therefore $v$ receives charge $1$ twice by (R1).
Thus, $\mu^*(v)=2\cdot2-6+2\cdot1=0$.
Note that there are no $3$-vertices by Lemma~\ref{lem:bal2:3-vx}, and every vertex is incident with at most $\lfloor{d(v)\over 2}\rfloor$ $3$-faces since there are no $4$-cycles in $G$.
If $v$ is a vertex with $d(v)\in\{4, 5, 6\}$, then $v$ sends charge $1$ at most $\lfloor{d(v)\over 2}\rfloor$ times by rule (R2).
Thus, $\mu^*(v)\geq 2d(v)-6-\lfloor{d(v)\over 2}\rfloor\geq 0$.

Now assume $v$ is a $7^+$-vertex. 
We will show that $v$ has nonnegative final charge by a weighting argument. 
Let $u_1, \ldots, u_{d(v)}$ be the neighbors of $v$ in some cyclic order. 
First give all neighbors of $v$ a weight of $1$. 
If $u_i$ is not a $2$-vertex, then split the weight of $1$ it received from $v$, and transfer weight $1\over 2$ to each of the two faces that are incident with $vu_i$; 
if $vu_i$ is incident with only one face, then transfer the entire weight of $1$ to this face. 
Now, every neighbor of $v$ that is a $2$-vertex and every face incident with $v$ that is not a terrible 3-face have weight at least the charge that they should receive from $v$ by the discharging rules.
Every terrible 3-face has weight at most $1$ short of the charge it should receive from $v$ by the discharging rules. 
Now give weight $1$ to each terrible 3-face incident with $v$. 
Since $v$ is incident with at most $d(v)-6$ terrible 3-faces by Lemma~\ref{lem:bal2:terrible}, and each neighbor of $v$ received weight $1$ initially, the total weight spend is at most $2d(v)-6$, which is exactly the initial charge of $v$. 
Thus, the total weight sent is no more than the initial charge of $v$, which proves that the final charge of $v$ is nonnegative.
\end{proof}



\section{Unbalanced $2$-partitions}\label{sec:unbal2}

In this section, we prove the following theorem:

\begin{thm}\label{thmm:unbal2}
\thmunbaltwo
\end{thm}

We will first show a necessary condition for cycle obstruction sets, and then show that it is sufficient afterwards. 

\begin{lem}\label{lem:unbal2:nec}
If a set $S$ of cycles is an obstruction set of unbalanced $2$-partitionable planar graphs, then either $\{C_3, C_4, C_6\} \subseteq S$ or $S$ contains all odd cycles.
\end{lem}
\begin{proof}
For a nonnegative integer $D$, a positive integer $l$, and a vertex $v$, recall that $H(D, l)$ from Section~\ref{sec:bal2} is not $(D, D)$-colorable and the only cycles in $H_1(D, l; v)$ have length either $4$ or $2l+1$. 
Therefore $H(D, l)$ is not $(0, D)$-colorable as well.
Therefore $S$ contains either $C_4$ or all odd cycles. 

Given a nonnegative integer $D$ and two vertices $x$ and $y$, let $F_1(D; x, y)$ be the graph that consists of $2D+1$ internally disjoint $x, y$-paths of length $3$. 
See Figure~\ref{fig:unbal2-tight} for an illustration of $F_1(D;x,y)$.
For an odd integer $l\geq 3$ and a vertex $v_1$, let $F_o(D, l; v_1)$ be the graph obtained from an odd cycle with vertices $v_1, \ldots, v_l$ by replacing each edge $v_iv_{i+1}$ with $F_1(D; v_i, v_{i+1})$ where $i$ is an odd integer at most $l$ (where $v_{l+1}$ is treated as $v_1$). 
Finally, obtain $F(D; l)$ from two disjoint copies of $F_o(D, l; v_1)$ and adding an edge between the two vertices that correspond to $v_1$. 
See Figure~\ref{fig:unbal2-tight} for an illustration of $F(D;5)$.

Now in any $(0, D)$-coloring of $F_1(D; x, y)$, it is easy to see that $x$ and $y$ cannot both receive the color $2$. 
The two cutvertices of $F(D; l)$ cannot both receive the color $1$ in any $(0, D)$-coloring, thus at least one cutvertex $v$ receives the color $2$. 
In the copy that corresponds to $F_o(D, l; v)$, either there is an edge with both endpoints colored with the color $1$ or there is a copy of $F_1(D; x, y)$ where both $x$ and $y$ receive the color $2$.
This shows that $F(D; l)$ is not $(0, D)$-colorable. 
It is not hard to see that the only cycles in $F(D; l)$ have length either $6$ or $2l+1$ where $l$ is an odd integer at least $3$.
Therefore $S$ contains either $C_6$ or all cycles of lengths $4k+3$ where $k$ is a positive integer. 

For an odd integer $l\geq 3$ and a vertex $v_1$, let $F_e(D, l; v_1)$ be the graph obtained from an odd cycle with vertices $v_1, \ldots, v_l$ by replacing each edge $v_iv_{i+1}$ with $F_1(D; v_i, v_{i+1})$ where $i$ is an even integer at most $l$. 
Finally, let $F'(D; l)$ be the graph obtained from a star with $D+2$ vertices by attaching a copy of $F_e(D, l; v)$ to each vertex $v$ of the star. 
See Figure~\ref{fig:unbal2-tight} for an illustration of  $F'(2;5)$.

As above, $x$ and $y$ cannot both receive the color $2$ in any $(0,D)$-coloring of $F_1(D; x, y)$. 
This implies that every cutvertex of $F'(D; l)$ must be colored with color $2$ in a $(0, D)$-coloring.
Yet, now there exists a cutvertex of that has $D+1$ neighbors colored with the color $2$, which implies that $F'(D; l)$ is not $(0, D)$-colorable. 
It is not hard to see that the cycles in $F'(D; l)$ have length either $6$ or $2l-1$ where $l$ is an odd integer at least $3$. 
Therefore $S$ contains either $C_6$ or all cycles of lengths $4k+1$ where $k$ is a positive integer. 

Let $T_0(D; x)$ be the graph obtained from $D+1$ pairwise disjoint $3$-cycles by identifying one vertex in each cycle into $x$. 
Now let $T(D)$ be the graph obtained from two copies of $T_0(D; x)$ and adding an edge between the two vertices corresponding to $x$. 
In any $(0, D)$-coloring of $T_0(D; x)$, the vertex $x$ must not receive color $2$ since it will have $D+1$ neighbors colored with $2$.
Yet, in $T(D)$, one of the two cutvertices, which corresponds to $x$ in a copy of $T_0(D; x)$, will receive color $2$.
This shows that $T(D)$ is not $(0, D)$-colorable, and it is easy to see that $T(D)$ contains only  $3$-cycles. 
Hence $S$ contains $C_3$. 

To sum up, the obstruction set of unbalanced $2$-partitionable planar graphs must contain $C_3$, and contains either $\{C_4, C_6\}$ or all odd cycles of length at least five. 
In other words, $S$ contains either $\{C_3, C_4, C_6\}$ or all odd cycles. 
\end{proof}

\begin{figure}[ht]
	\begin{center}
  \includegraphics[scale=0.8]{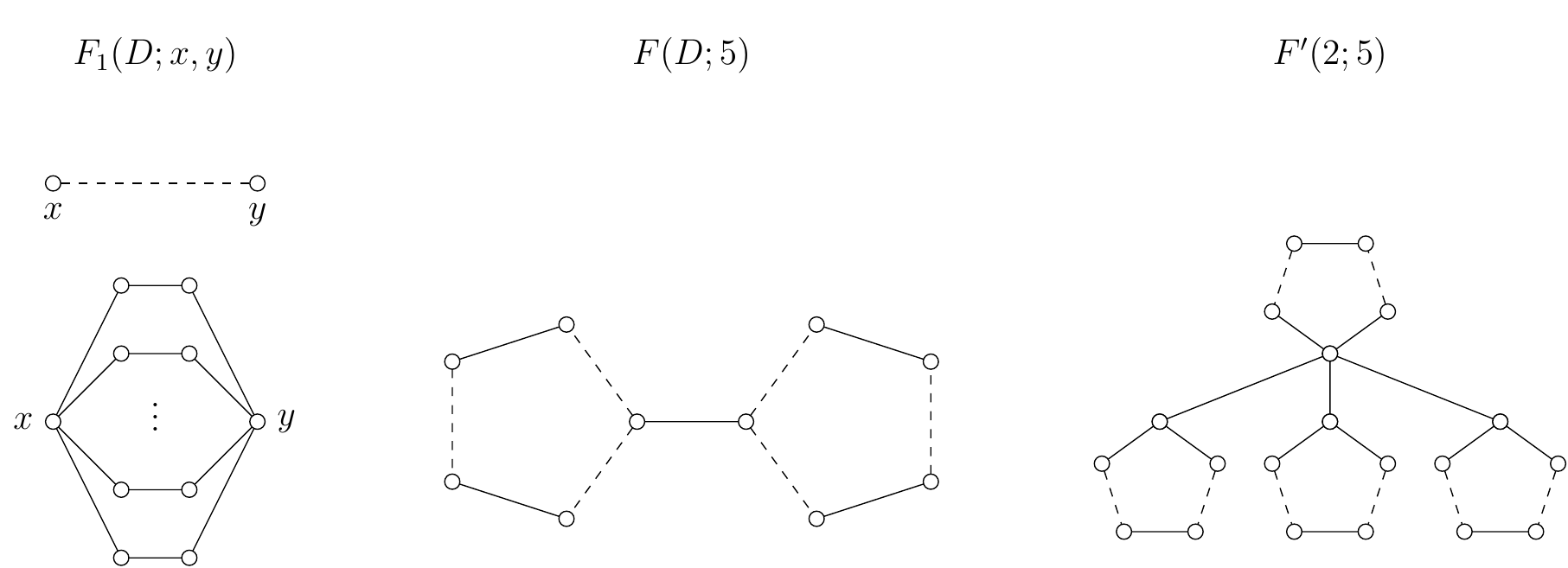}
  \caption{Graphs that are not $(0, D)$-colorable}
  
\label{fig:unbal2-tight}
	\end{center}
\end{figure}

If either $C_4$ or $C_6$ is not in an obstruction set $S$ of unbalanced $2$-partitionable planar graphs, then all odd cycles must be in $S$.
This implies that the graph is bipartite and $(0, 0)$-colorable, and hence it is unbalanced $2$-partitionable. 
The following theorem shows that $\{C_3, C_4, C_6\}$ is an obstruction set of unbalanced $2$-partitionable planar graphs. 
Note that Lemma~\ref{lem:unbal2:nec} and Theorem~\ref{thm:unbal2} imply Theorem~\ref{thmm:unbal2}.

\begin{thm}\label{thm:unbal2}
A planar graph with no $3$-, $4$-, $6$-cycles is $(0, 45)$-colorable. 
\end{thm}

In this section, let $G$ be a counterexample to Theorem~\ref{thm:unbal2} with the minimum number of vertices. 
Also, fix a plane embedding of $G$. 
It is easy to see that $G$ is connected and has no $1$-vertices.
From now on, given a (partially) $(0, 45)$-colored graph, we will let $a$ and $b$ be the two colors where $b$ is the color class allowed to have maximum degree at most $45$, and we say a vertex colored with $b$ is {\it saturated} if it already has 45 neighbors colored with $b$.

\subsection{Structural lemmas}

\begin{lem}\label{lem:unbal2:vx-degree}
Any $46^-$-vertex is adjacent to a $47^+$-vertex. 
\end{lem}
\begin{proof}
Suppose to the contrary that a $46^-$-vertex $v$ is adjacent to only $46^-$-vertices. 
Since $G-v$ is a graph with fewer vertices than $G$, there is a $(0, 45)$-coloring $\varphi$ of $G-v$; choose $\varphi$ that maximizes the number of neighbors of $v$ with the color $a$.
At least one neighbor of $v$ has color $a$, since otherwise we can extend $\varphi$ to all of $G$ by coloring $v$ with color $a$. 
Also, every neighbor of $v$ colored $b$ has a neighbor in $G-v$ with the color $a$, otherwise it can be recolored by $a$ and violates the choice of $\varphi$. 
Since each neighbor $u$ of $v$ has at most $45$ neighbors in $G-v$, $u$ has at most $44$ neighbors with the color $b$ in $G-v$. 
So no neighbor of $v$ is saturated.
Hence we can extend $\varphi$ to $G$ by coloring $v$ with color $b$.
This contradicts that $G$ is a counterexample, and thus proves the claim. 
\end{proof}

Since $G$ has no $3$-cycles and no $4$-cycles, every $5$-face is bounded by a cycle. 
A {\it bad face} is a $5$-face $f$ where the degrees of the vertices on a boundary walk is as in Figure~\ref{fig:unbal2-bad}.

\begin{figure}[ht]
	\begin{center}
  \includegraphics[scale=1]{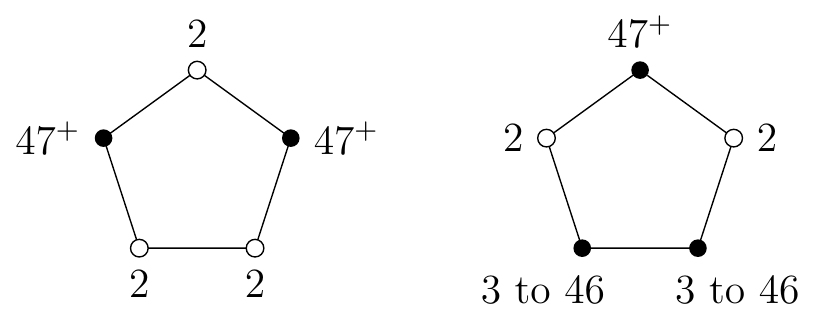}
  \caption{Bad faces}
  \label{fig:unbal2-bad}
	\end{center}
\end{figure}

\begin{lem}\label{lem:unbal2:bad-faces}
Any $2$-vertex cannot be incident with two bad faces. 
\end{lem}
\begin{proof}
Suppose to the contrary that a $2$-vertex $v$ is incident with two bad faces where $x, v,y,v_1,v_2$ and $x, v, y,u_1,u_2$ are vertices, in this order, of boundary walks of the two bad faces. 
If $v_1=u_2$ (or $v_2=u_1$), then $G$ contains a $3$-cycle $xv_1v_2$ (or $yv_1v_2$), which is a contradiction.
If $v_1=u_1$ (or $v_2=u_2$), then $G$ has a $4$-cycle $v_1v_2xu_2$ (or $yv_1v_2u_1$), which is again a contradiction.
Therefore, $\{v_1, v_2\}\cap\{u_1, u_2\}=\emptyset$, and this implies that $G$ contains a $6$-cycle with vertices $x, v_2, v_1, y, u_1, u_2$, which is a contradiction. 
\end{proof}

\begin{lem}\label{lem:unbal2:bad-faces1}
Any  $47^+$-vertex $v$ is incident with at most $\lfloor{d(v)\over 2}\rfloor$ bad faces. 
\end{lem}
\begin{proof}
Suppose to the contrary that some $47^+$-vertex $v$ is incident with at least $\lfloor {d(v)\over2} \rfloor+1$ bad faces. 
Then some edge $e$ incident with $v$ is contained in two different bad faces. 
By the definition of bad faces, the end of $e$ other than $v$ has degree $2$. 
So this $2$-vertex is incident with two different bad faces, contradicting Lemma~\ref{lem:unbal2:bad-faces}.
\end{proof}

\subsection{Discharging}


We now define the initial charge at each vertex and each face.
For every $v\in V(G)$, let $\mu(v)=2d(v)-6$ and for every face $f\in F(G)$, let $\mu(f)=d(f)-6$. 
The total initial charge is negative since
\begin{align*}
\sum_{z\in V(G)\cup F(G)} \mu(z)
	=\sum_{v\in V(G)} (2d(v)-6)+\sum_{f\in F(G)} (d(f)-6) 
	=-6|V(G)|+6|E(G)|-6|F(G)|
	=-12	<0.
\end{align*}
The last equality holds by Euler's formula.

Recall that a bad face is a $5$-face and there are two non-adjacent $2$-vertices on that face. 
For each face $f$, let $W_f$ be a canonical boundary walk of $f$.
Recall that for any face $f$ and vertex $v$, $k_{f,v}$ is the number of triples $(e,v,e')$ such that $e,e' \in E(G)$ and $eve'$ is a subwalk of $W_f$.

Here are the discharging rules:

\begin{enumerate}[(R1)]

\item Let $v$ be a $47^+$-vertex.

\begin{enumerate}[({R1}A)]
\item $v$ sends charge $1$ to each adjacent vertex. 
\item $v$ sends charge $1$ to each incident bad face.
\item $v$ sends charge $\frac{3}{4}k_{f,v}$ to each incident face $f$ that is not bad. 
\end{enumerate}

\item Let $v$ be a vertex where $d(v)\in\{3, \ldots, 46\}$. 

\begin{enumerate}[(R2A)]
\item $v$ sends charge $1\over 2$ to each adjacent $2$-vertex. 
\item $v$ sends charge $t \over 2$ to each incident face $f$, where $t$ is the number of triples $(x,v,y)$ such that $x,y \in V(G)$, $xvy$ is a subpath in $W_f$, and either both $d(x),d(y)$ are at least $47$, or both $d(x),d(y) \in \{3, \ldots,46\}$. 
\item $v$ sends charge $t \over 4$ to each incident face $f$, where $t$ is the number of triples $(x,v,y)$ such that $x,y \in V(G)$, $xvy$ is a subpath in $W_f$, $d(x)=2$ and $d(y) \in \{3, \ldots,46\}$. 
\end{enumerate}

\item Let $f$ be a face. 

\begin{enumerate}[(R3A)]
\item $f$ sends charge $\frac{1}{2} k_{f,v}$ to each incident $2$-vertex $v$ that is adjacent to another $2$-vertex.
\item $f$ sends charge $\frac{1}{4} k_{f,v}$ to each incident $2$-vertex $v$ that is adjacent to a vertex $y$ with $d(y)\in\{3, \ldots, 46\}$. 
\end{enumerate}

\end{enumerate}

The discharging rule (R1) shows how a $47^+$-vertex distributes its initial charge, (R2) shows how a vertex with degree in $\{3, \ldots, 46\}$ sends charge, and (R3) shows how a face sends its charge. 
Note that by Lemma~\ref{lem:unbal2:vx-degree}, a face does not send charge to a $2$-vertex via both (R3A) and (R3B).
See Figure~\ref{fig:unbal2-rules} for an illustration of the discharging rules. 

\begin{figure}[h]
	\begin{center}
		\includegraphics{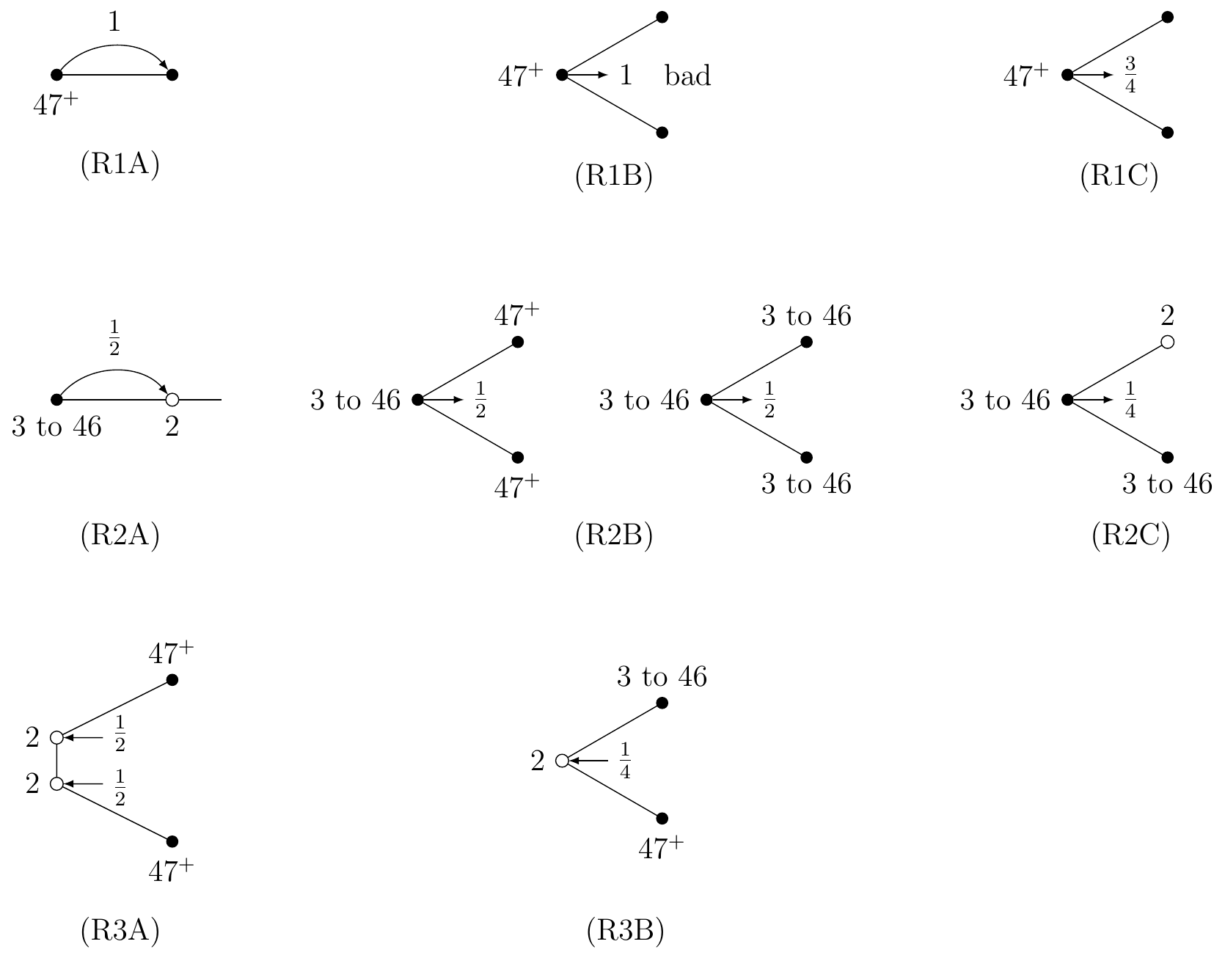}
	\end{center}
  \caption{Discharging rules.}
  \label{fig:unbal2-rules}
\end{figure}

The rest of this section will prove that the final charge $\mu^*(z)$ is nonnegative for each $z\in V(G)\cup F(G)$. 

\begin{claim}
Every vertex has nonnegative final charge. 
\end{claim}
\begin{proof}
Assume $v$ is a $2$-vertex. 
If $v$ is adjacent to two $47^+$-vertices, then $v$ receives charge $1$ from each of its neighbors by (R1A). 
Thus, $\mu^*(v)=-2+2\cdot{1}=0$. 
Note that $v$ is adjacent to at least one $47^+$-vertex by Lemma~\ref{lem:unbal2:vx-degree}.
If $v$ is adjacent to another $2$-vertex, then $v$ receives charge $2 \cdot \frac{1}{2}$ from the faces incident with $v$ by (R3A).
Thus, $\mu^*(v)=-2+2\cdot{1\over 2}+1=0$. 
Otherwise, $v$ is adjacent to a vertex of degree from $3$ to $46$, which sends charge $1\over 2$ to $v$ by (R2A). 
Also, $v$ receives charge $2 \cdot \frac{1}{4}$ from the faces incident with $v$ by (R3B). 
Thus, $\mu^*(v)=-2+2\cdot{1\over 4}+{1\over 2}+1=0$. 

Assume $v$ is a $47^+$-vertex. 
By (R1A), $v$ sends charge at most $d(v)$ to its adjacent vertices in total. 
By Lemma~\ref{lem:unbal2:bad-faces1}, $v$ is incident with at most $\lfloor{d(v)\over 2}\rfloor$ bad faces. 
Since $v$ sends charge $1$ to each of its incident bad faces by (R1B) and sends charge $3\over 4$ to each of its incident faces that are not bad by (R1C), the final charge $\mu^*(v)$ is at least $2d(v)-6-d(v)-\lfloor{d(v)\over 2}\rfloor-{3\over4}\cdot \lceil{d(v)\over 2}\rceil$, which is nonnegative since $d(v)\geq 47$. 

Assume $d(v)\in\{4, \ldots, 46\}$. 
We will show that $v$ has nonnegative final charge by using a weighting argument. 
Let $u_1, \ldots, u_{d(v)}$ be the neighbors of $v$ in some cyclic order. 
First give all neighbors of $v$ a weight of $1\over 2$.
If $u_i$ is not a $2$-vertex, then split the weight of $1\over 2$ it received from $v$, and transfer weight $1\over 4$ to each of the two faces that are incident with $vu_i$ (if $vu_i$ is incident with only one face, then transfer weight $1 \over 2$ to this face). 
Now, every $2$-vertex adjacent to $v$ and every face that is incident with $v$ have weight 
equal to the charge sent from $v$ in the discharging rules.
So the total charge sent from $v$ is at most the weight sent from $v$. 
Since $v$ has charge $2d(v)-6\geq {d(v)\over 2}$ when $d(v)\geq 4$, $v$ has nonnegative final charge. 

Assume $v$ is a $3$-vertex. 
If $v$ is adjacent to at least two $47^+$-vertices, which each sends charge $1$ to $v$ by (R1A), then $v$ is adjacent to at most one $2$-vertex. 
Thus, $\mu^*(v)\geq 0+2-4\cdot{1\over 2}=0$. 
If $v$ is adjacent to exactly one $47^+$-vertex, then $v$ sends charge at most $1\over 2$ to at most twice according to the discharging rules. 
In either case, $\mu^*(v)\geq 0+1-2\cdot{1\over 2}=0$. 
\end{proof}

\begin{claim}
Each $7^+$-face $f$ has nonnegative final charge.
\end{claim}
\begin{proof}
We will show that $f$ has nonnegative final charge by using a weighting argument. 
Pull weight $\frac{3}{4}k_{f,v}$ from each $47^+$-vertex $v$ on $f$ (note that this corresponds to (R1C)), and transfer weight $\frac{3}{8}k_{f,v}$ to each $2$-vertex on $f$ that is adjacent to $v$. 
Each $2$-vertex on $f$ receives weight at least $\frac{3}{8}k_{f,v}$, since it must be adjacent to a $47^+$-vertex, which is on $f$, by Lemma~\ref{lem:unbal2:vx-degree}. 
Now if $f$ sends an additional weight of $\frac{1}{8}k_{f,v}$ to each $2$-vertex on $f$, then (R3) is satisfied. 
By Lemma \ref{lem:unbal2:vx-degree}, there cannot be three consecutive $2$-vertices on a boundary walk of $f$, so it follows that $\sum k_{f,v} \leq \lfloor \frac{2}{3}d(f) \rfloor$, where the sum is over all 2-vertices incident with $f$. 
Therefore, $\mu^*(f) \geq d(f)-6-\frac{1}{8} \sum k_{f,v} \geq d(f)-6-\frac{1}{8} \lfloor \frac{2}{3}d(f) \rfloor>0$ when $d(f) \geq 7$, where the sum is over all 2-vertices incident with $f$. 
\end{proof}

Note that there is no $6$-face since $G$ has no $1$-vertex and no $3$-, $4$-, $6$-cycles. 

\begin{claim}
Each $5$-face $f$ has nonnegative final charge.
\end{claim}
\begin{proof}
Since $G$ has no $1$-vertex, every 5-face is bounded by a cycle.
Let $v_1, v_2, v_3, v_4, v_5$ be the vertices of $f$ in some cyclic order. 

Assume $f$ is incident with at most one $2$-vertex, and assume $v_1$ is the $2$-vertex, if any. 
Note that $f$ sends charge $1\over 4$ to $v_1$ by (R3B) if it is a $2$-vertex. 
If at least two of $v_2, \ldots, v_5$ are $47^+$-vertices, then $f$ receives charge $3\over 4$ from each one by (R1C), thus, $\mu^*(f)\geq -1+2\cdot{3\over 4}-{1\over 4}>0$. 
If exactly one of $v_2, \ldots, v_5$ is a $47^+$-vertex, then without loss of generality we may assume it is $v_2$ by Lemma~\ref{lem:unbal2:vx-degree}.
Now $v_4$ sends charge $1\over 2$ to $f$ by (R2B), thus, $\mu^*(f)\geq -1+{3\over 4}+{1\over 2}-{1\over 4}=0$. 
If none of $v_2, \ldots, v_5$ is a $47^+$-vertex, then $v_1$ cannot be a $2$-vertex. 
Since both $v_3$ and $v_4$ send charge $1\over 2$ by (R2B), it follows that $\mu^*(f)\geq -1+2\cdot{1\over 2}=0$. 

Assume $f$ is incident with at least two $2$-vertices where two of them, say $v_2$ and $v_3$, are adjacent to each other. 
Note that $f$ sends charge $1\over 2$ to each of $v_2$ and $v_3$ by (R3A). 
By Lemma~\ref{lem:unbal2:vx-degree}, both $v_1$ and $v_4$ must be $47^+$-vertices. 
If $v_5$ is not a $2$-vertex, then $f$ is not a bad face, and $v_1, v_4, v_5$ send charge $3\over 4$, $3\over 4$, at least $1\over 2$, respectively, by (R1C) and (R2B).
Thus, $\mu^*(f)\geq -1+2\cdot{3\over 4}+{1\over 2}-2\cdot{1\over 2}=0$. 
If $v_5$ is a $2$-vertex, then $f$ is a bad face, and both $v_1$ and $v_4$ send charge $1$ each to $f$ by (R1B).
Thus, $\mu^*(f)\geq -1+2\cdot{1}-2\cdot{1\over 2}=0$. 

If $f$ is incident with at least two $2$-vertices and where no pair is nonadjacent, then $f$ is incident to exactly two $2$-vertices by Lemma~\ref{lem:unbal2:vx-degree}.
Thus, the only remaining case is when $f$ is incident with exactly two nonadjacent $2$-vertices, say $v_1$ and $v_3$.
Note that $f$ sends charge $1\over 4$ to each of $v_1$ and $v_3$ by (R3B). 
If $f$ is incident with at least two $47^+$-vertices, which each sends charge at least $3\over 4$ to $f$ by (R1), then $\mu^*(f)\geq -1+2\cdot{3\over 4}-2\cdot{1\over 4}=0$. 
Now $f$ must be incident with exactly one $47^+$-vertex because $f$ is incident with a $2$-vertex, and by Lemma~\ref{lem:unbal2:vx-degree} we know that $v_2$ must be the $47^+$-vertex. 
It follows that $d(v_4), d(v_5)\in\{3, \ldots, 46\}$, and therefore $f$ is a bad face. 
Now, $v_2, v_4, v_5$ send charge $1, {1\over 4}, {1\over 4}$, respectively, to $f$ by (R1B) and (R2C).
Thus, $\mu^*(f)\geq -1+1+2\cdot{1\over 4}-2\cdot{1\over 4}=0$. 

\end{proof}



\section{Unbalanced $3$-partitions}\label{sec:unbal3}

In this section, we prove the following theorem:

\begin{thm}\label{thmm:unbal3}
\thmunbalthree
\end{thm}

We will first show a necessary condition for cycle obstruction sets, and then show that it is sufficient afterwards. 

\begin{lem}\label{lem:unbal3:nec}
If a set $S$ of cycles is an obstruction set of unbalanced $3$-partitionable planar graphs, then either $C_3\in S$ or $C_4\in S$.
\end{lem}
\begin{proof}
Let $X_0(D; v)$ be the graph that is obtained from starting with $D+1$ pairwise disjoint copies of $K_4$ and picking one vertex from each copy of $K_4$ and identifying them into $v$. 
Now let $X(D)$ be the graph obtained from three copies of $X_0(D; v)$ and adding three edges between the three vertices that correspond to $v$. 
See Figure~\ref{fig:unbal3-tight} for an illustration of $X_0(2; v)$ and $X(2)$. 
Now in any $(0, 0, D)$-coloring of $X_0(D; v)$, the vertex $v$ cannot receive the color $3$. 
This is because each copy of $K_4-v$ must contain a vertex colored with $3$, and since there are $D+1$ copies, $v$ has $D+1$ neighbors with the same color, which is a contradiction. 
However, in any $(0, 0, D)$-coloring of $X(D)$, one vertex $v$ of the three cutvertices must receive the color $3$, and this shows that $X(D)$ is not $(0, 0, D)$-colorable. 
It is not hard to see that the only cycles in $X(D)$ have length either $3$ or $4$.
\end{proof}

\begin{figure}[ht]
	\begin{center}
  \includegraphics[scale=0.8]{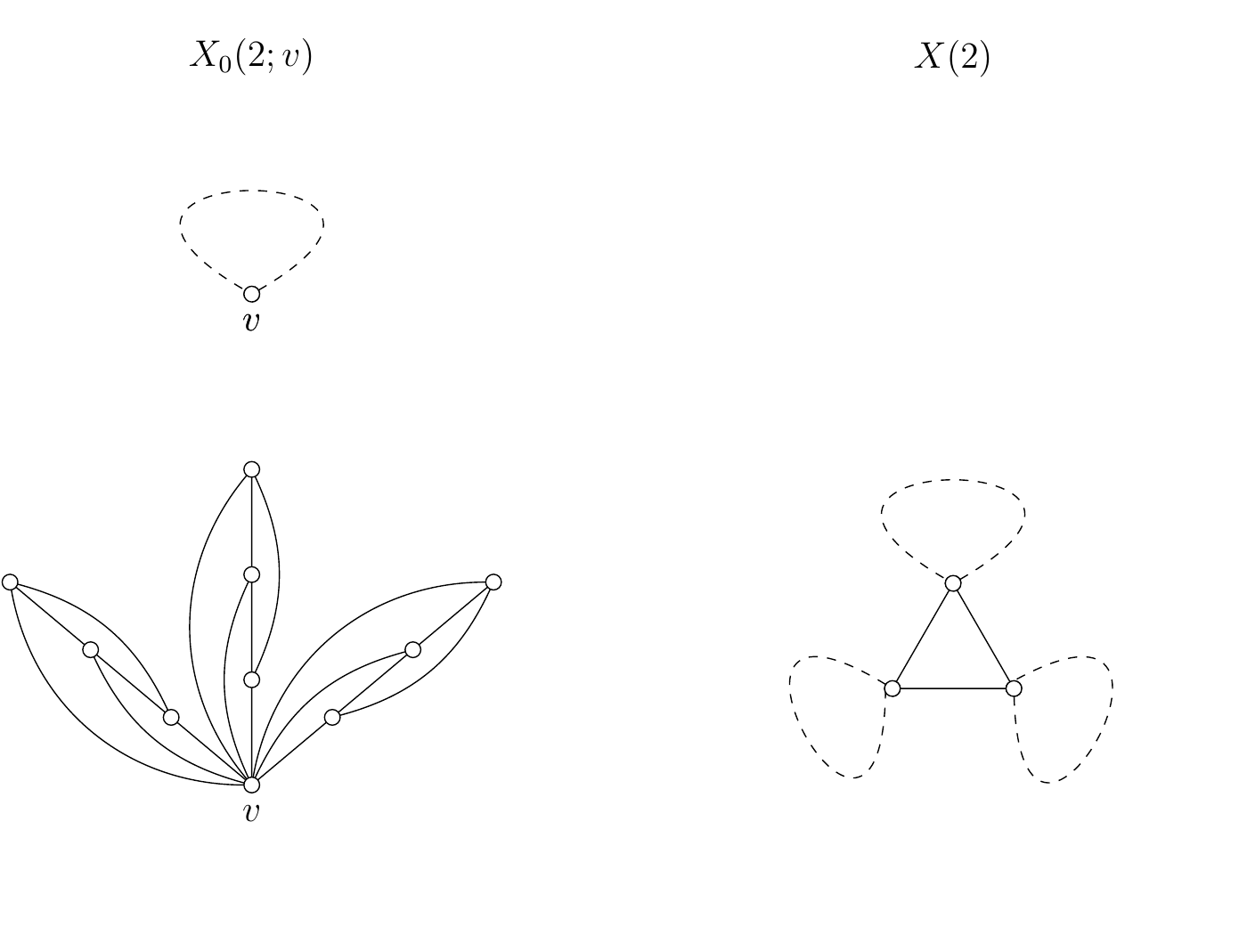}
  \caption{Graphs that are not $(0, 0, D)$-colorable}
  \label{fig:unbal3-tight}
	\end{center}
\end{figure}

If a planar graph does not contain $3$-cycles, then it is $3$-colorable, which is equivalent to $(0, 0, 0)$-colorable, by Gr\"otzsch's Theorem~\cite{1959Gr}, and thus it is unbalanced $3$-partitionable. 
This means that $\{C_3\}$ is an inclusion-wise  minimal obstruction set of unbalanced 3-partitionable planar graphs.
The remaining of this section proves Theorem \ref{thm:unbal3} below, which states that planar graphs with no $4$-cycles are unbalanced $3$-partitionable. 
Note that Lemma~\ref{lem:unbal3:nec} and Theorem~\ref{thm:unbal3} imply Theorem~\ref{thmm:unbal3}.

\begin{thm}\label{thm:unbal3}
Any planar graph with no $4$-cycles is $(0, 0, 117)$-colorable. 
\end{thm}

In this section, let $G$ be a counterexample to Theorem~\ref{thm:unbal3} with the minimum number of vertices. 
Also, fix a plane embedding of $G$. 
It is easy to see that $G$ is connected and there are no $2^-$-vertices in $G$.

From now on, given a (partially) $(0, 0, 117)$-colored graph, we will let $a$, $b$, $c$ be the color of the color class that is allowed to have maximum degree at most $0$, $0$, $117$, respectively, and we say a vertex colored with $c$ is {\it saturated} if it already has $117$ neighbors colored with $c$.

\subsection{Structural lemmas}

\begin{lem}\label{lem:unbal3:vx-degree}
A $119^-$-vertex is adjacent to a $120^+$-vertex. 
\end{lem}
\begin{proof}
Suppose to the contrary that a $119^-$-vertex $v$ is adjacent to only $119^-$-vertices. 
Since $G-v$ is a graph with fewer vertices than $G$, there is a $(0, 0, 117)$-coloring $\varphi$ of $G-v$. 
We further assume that $\varphi$ minimizes the number of neighbors of $v$ colored with $c$.
If there exists a neighbor $u$ of $v$ in $G$ such that $\varphi(u)=c$ and $u$ is saturated, then at most one neighbor of $u$ in $G-v$ has a color in $\{a, b\}$, so we can recolor $u$ to be a color in $\{a,b\}$ that does not appear in its neighborhood in $G-v$, contradicting the minimality of $\varphi$.
Hence no neighbor $u$ of $v$ with color $c$ is saturated.
If no neighbor of $v$ is colored with a color in $\{a, b\}$, then we can extend $\varphi$ to all of $G$ by coloring $v$ with a color in $\{a,b\}$ that does not appear in the neighborhood of $v$ in $G$, contradicting that $G$ is a counterexample.
So both $a$ and $b$ appear in the neighborhood of $v$ in $G$, and thus there are at most $117$ neighbors of $v$ colored with $c$.
Since no neighbor of $v$ with color $c$ is saturated, we can extend $\varphi$ to all of $G$ by coloring $v$ with color $c$, a contradiction.
\end{proof}


\begin{lem}\label{lem:unbal3:ext-coloring2}
Let $X$ be a set of $3$-vertices of $G$ such that the subgraph of $G$ induced on $X$ is a path $v_1v_2\ldots v_{k}$ where $k\geq 2$.
If $x$ and $y$ are the neighbors of $v_k$ in $G-X$,
then $c \in \{\varphi(x),\varphi(y)\}$ and $\varphi(x) \neq \varphi(y)$ for every $(0,0,117)$-coloring $\varphi$ of $G-X$.
Moreover, the vertex in $\{x, y\}$ that receives the color $c$ must be a $116^+$-vertex.
\end{lem}
\begin{proof}
Let $\varphi$ be a $(0,0,117)$-coloring of $G-X$ and let $x'$ and $y'$ be the neighbors of $v_1$ in $G-X$.
For each integer $i$ with $2 \leq i \leq k-1$, let $u_i$ be the vertex in $G-X$ adjacent to $v_i$.
First we extend $\varphi$ to a $(0,0,117)$-coloring of $G-v_k$ by defining $\varphi(v_1) \in \{a,b,c\}\setminus\{\varphi(x'),\varphi(y')\}$ and $\varphi(v_i) \in \{a,b,c\}-\{\varphi(v_{i-1}),\varphi(u_i)\}$ for each $i\in\{2, \ldots, k-1\}$.
If $\varphi(x)=\varphi(y)$, then we can extend $\varphi$ to be a $(0,0,117)$-coloring of $G$ by further defining $\varphi(v_k)$ to be an element in $\{a,b,c\}\setminus\{\varphi(v_{k-1}),\varphi(x)\}$.
This proves $\varphi(x) \neq \varphi(y)$.
If either $c \not \in \{\varphi(x),\varphi(y)\}$ or the vertex in $\{x, y\}$ with the color $c$ is a $115^-$-vertex, then by defining $\varphi(v_k)=c$,
we extended $\varphi$ to a $(0,0,117)$-coloring of $G$ since the degree of $v_{k-1}$ is $3$.
Therefore, $c \in \{\varphi(x),\varphi(y)\}$.
\end{proof}

\begin{lem}\label{lem:unbal3:ext-coloring3}
Let $X$ be a set of $3$-vertices of $G$ such that the subgraph of $G$ induced on $X$ is a path $v_1v_2\ldots v_{2k}$ on an even number of vertices.
Let $u_i$ be a neighbor of $v_i$ in $G-X$ for each $i$ with $1 \leq i \leq 2k$.
Let $x$ and $y$ be the neighbor of $v_1$ and $v_{2k}$, respectively in $G-X$ other than $u_1$ and $u_{2k}$.
If there exists a $(0,0,117)$-coloring $\varphi$ of $G-X$ such that $\varphi(u_i)=c$ for every $i$ with $1 \leq i \leq 2k$, then $\varphi(x)=\varphi(y)$.
\end{lem}
\begin{proof}
By Lemma \ref{lem:unbal3:ext-coloring2}, we may assume $\varphi(x) \neq \varphi(u_1)=c$.
Define $\varphi(v_1)=\{a,b\}\setminus\{\varphi(x)\}$ and $\varphi(v_i)=\{a,b\}\setminus\{\varphi(v_{i-1})\}$ for every $2 \leq i \leq 2k$.
Since $\lvert X \rvert$ is even, $\varphi(v_1) \neq \varphi(v_{2k})$.
That is, $\varphi(v_{2k})=\varphi(x)$.
As this must not extend $\varphi$ to be a $(0,0,117)$-coloring of $G$, $\varphi(y)=\varphi(v_{2k})$.
Therefore $\varphi(y)=\varphi(x)$.
\end{proof}

A face $f$ is {\it annoying} if exactly one vertex incident with $f$ is a $120^+$-vertex and all other vertices incident with $f$ are $3$-vertices. 
We say that two faces are {\it adjacent} if they share at least one edge.

\begin{lem}\label{lem:unbal3:annoying}
If an annoying $5$-face $f$ is adjacent to only annoying $3$-faces and annoying $5$-faces, then $f$ is adjacent to at most two $3$-faces. 
\end{lem}
\begin{proof}
Let $f=wx'xyy'$ where $w$ is the $120^+$-vertex on $f$ and $x',x,y,y'$ are all $3$-vertices.
For $e\in\{wx', x'x, xy, yy', y'w\}$, let $f_e$ be the face incident with $e$ other than $f$. 

Suppose to the contrary that $f$ is adjacent to three annoying $3$-faces. 
Since $G$ has no $4$-cycles, two $3$-faces cannot share an edge, and two faces incident with the same $3$-vertex must share an edge. 
Since $x', x, y, y'$ are all $3$-vertices, this implies that $f_{xy}, f_{wx'}, f_{y'w}$ must be the annoying $3$-faces adjacent to $f$. 

Let $f_{xy}=xyz$ so that $z$ is a $120^+$-vertex and the common neighbor of $x$ and $y$.
Also, let $f_{wx'}=wx'x$ and $f_{y'w}=wy'y_1$ so that $x_1$ and $y_1$ is the common neighbor of $w, x'$ and $w,y'$, respectively, which must be a $3$-vertex. 
Note that $z, x_1, y_1$ must be all distinct since otherwise that would imply the existence of a $4$-cycle. 

Let $z_1$ be the common neighbor of $z$ and $x_1$. 
Since $f_{xx'}$ is an annoying $5$-face, $z_1$ must be a $3$-vertex. 
Also, $z_1 \not \in \{z, x, y, x', y', x_1, y_1, w\}$ since there are no $4$-cycles. 
Let $z_2$ be the common neighbor of $z$ and $y_1$.
Similarly, $z_2$ is a $3$-vertex and $z_2\not\in\{z, x, y, x', y', x_1, y_1, w\}$. 
Note that $z_1 \neq z_2$, since $z$ has degree at least $120$, $f_{xx'}$ and $f_{yy'}$ are 5-faces and $xyz$ is a 3-face.
Note that the subgraph of $G$ induced on $\{x_1,x',x,y,y',y_1\}$ is a path.
See Figure~\ref{fig:unbal3:annoying} for an illustration.




Suppose that $z_1z_2$ is not an edge of $G$.
Set $H=(G-\{x, y, x', y'\})\cup x_1y_1$. 
Note that $H$ is still a plane graph with no $4$-cycles, since $z_1z_2$ is not an edge. 
Since $H$ is a graph with fewer vertices than $G$, there is a $(0, 0, 117)$-coloring $\varphi$ of $H$. 
Note that $\varphi$ is a $(0,0,117)$-coloring of $G-\{x,y,x',y'\}$.
If $\varphi(z) \neq c$, then let 
\begin{align*}
\varphi(x)&=c&
\varphi(x')&\in\{a, b, c\}\setminus\{\varphi(x_1), \varphi(w)\}\\
\varphi(y')&\in\{a, b, c\}\setminus\{\varphi(w), \varphi(y_1)\}&
\varphi(y)&\in\{a, b, c\}\setminus\{\varphi(z), \varphi(y')\}
\end{align*}
 to extend $\varphi$ to all of $G$. 
Hence $\varphi(z)=c$.
Since $w$ is a $120^+$-vertex, by Lemma \ref{lem:unbal3:ext-coloring2}, $\varphi(w)=c$ and $\{\varphi(x_1),\varphi(y_1)\} \subseteq \{a,b\}$.
Since $\varphi(z)=\varphi(w)=c$, Lemma \ref{lem:unbal3:ext-coloring3} implies that $\varphi(x_1)=\varphi(y_1)$.
However, $\{\varphi(x_1),\varphi(y_1)\} \subseteq \{a,b\}$ and $x_1y_1$ is an edge of $H$, so $\varphi(x_1) \neq \varphi(y_1)$, a contradiction.

Therefore, $z_1z_2$ is an edge of $G$.
Since $G-\{x, y, x', y', x_1, y_1\}$ is a graph with fewer vertices than $G$, there exists a $(0, 0, 117)$-coloring $\varphi$ for $G-\{x, y, x', y', x_1, y_1\}$. 
If $\varphi(z) \neq c$, then let
\begin{align*}
\varphi(x)&=c&
\varphi(x_1)&\in\{a, b, c\}\setminus\{\varphi(z_1), \varphi(w)\}\\
\varphi(x')&\in\{a, b, c\}\setminus\{\varphi(x_1), \varphi(w)\}&
\varphi(y_1)&\in\{a, b, c\}\setminus\{\varphi(z_2), \varphi(w)\}\\
\varphi(y')&\in\{a, b, c\}\setminus\{\varphi(w), \varphi(y_1)\}&
\varphi(y)&\in\{a, b, c\}\setminus\{\varphi(z), \varphi(y')\}
\end{align*}
to extend $\varphi$ to all of $G$, which is a contradiction. 
Hence $\varphi(z)=c$.
Since $w$ is a $120^+$-vertex, by Lemma \ref{lem:unbal3:ext-coloring2}, $\varphi(w)=c$ and $\{\varphi(z_1),\varphi(z_2)\} \subseteq \{a,b\}$.
Since $\varphi(z)=\varphi(w)=c$, Lemma \ref{lem:unbal3:ext-coloring3} implies that $\varphi(z_1)=\varphi(z_2)$.
However, $\{\varphi(z_1),\varphi(z_2)\} \subseteq \{a,b\}$ and $z_1z_2$ is an edge of $G-\{x, y, x', y', x_1, y_1\}$, so $\varphi(z_1) \neq \varphi(z_2)$, a contradiction.
\end{proof}

\begin{figure}[h]
	\begin{center}
		\includegraphics{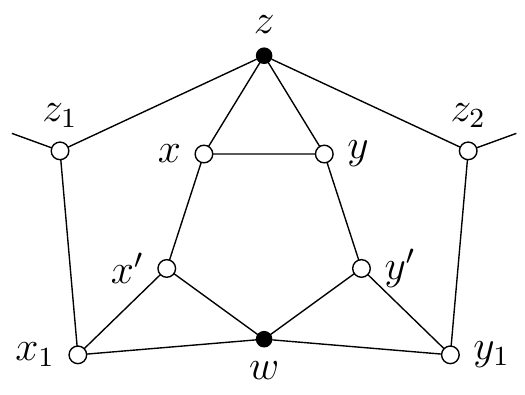}
	\end{center}
  \caption{Figure for Lemma~\ref{lem:unbal3:annoying}}
  \label{fig:unbal3:annoying}
\end{figure}

\subsection{Discharging}


We now define the initial charge at each vertex and each face.
For every $z\in V(G)\cup F(G)$, let $\mu(z)=d(z)-4$. 
The total initial charge is negative since
\begin{align*}
\sum_{z\in V(G)\cup F(G)} \mu(z)
	=\sum_{v\in V(G)} (d(v)-4)+\sum_{f\in F(G)} (d(f)-4) 
	=-4|V(G)|+4|E(G)|-4|F(G)|
	=-8
	<0.
\end{align*}
The last equality holds by Euler's formula.


Here are the discharging rules:

\begin{enumerate}[(R1)]

\item Each $5^+$-face $f$ sends charge $\frac{k_{f,v}}{r}(d(f)-4)$ to each incident 3-vertex $v$, where $r=\sum k_{f,u}$ and the sum is over all 3-vertices $u$ incident with $f$. 

\item Let $v$ be a $120^+$-vertex.

\begin{enumerate}[({R2}A)]
\item $v$ sends charge $2\over 3$ to each neighbor. 
\item $v$ sends charge $3\over 5$ to each incident $3$-face.
\end{enumerate}

\item Each vertex $v$ where $d(v)\in\{4, \ldots, 119\}$ sends charge $1\over 3$ to each incident $3$-face. 

\item Each $3$-vertex that is not incident with a $3$-face sends charge $1\over 15$ to each adjacent $3$-vertex. 
\end{enumerate}

The discharging rule (R1) shows how a face distributes its initial charge, (R2) shows how a $120^+$-vertex sends charge, (R3) shows how a vertex with degree in $\{4, \ldots, 119\}$ sends charge, and (R4) shows how a $3$-vertex that is not incident with a $3$-face sends charge to an adjacent $3$-vertex. 
See Figure~\ref{fig:unbal3-rules} for an illustration of the discharging rules. 

\begin{figure}[h]
	\begin{center}
		\includegraphics{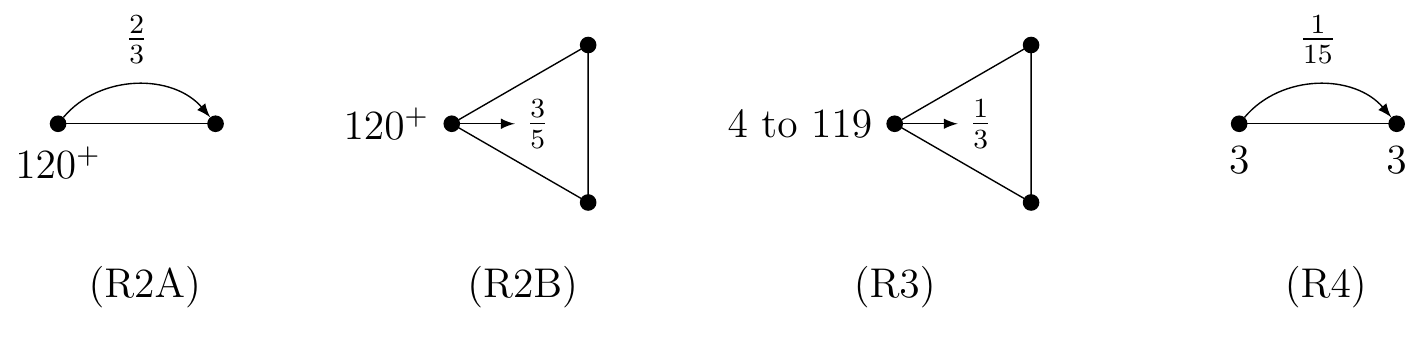}
	\end{center}
  \caption{Discharging rules}
  \label{fig:unbal3-rules}
\end{figure}

The rest of this section will prove that the sum of the final charge $\mu^*(z)$ is nonnegative for $z\in V(G)\cup F(G)$. 
Note that every $5^+$-face has nonnegative final charge since it only distributes its initial charge, which is positive. 
There are no $4$-faces since there are no $4$-cycles, and each edge is incident with at most one $3$-face since there are no $4$-cycles. 
We will first show that each $4^+$-vertex has nonnegative final charge.
Then, instead of counting $3$-vertices and $3$-faces separately, we will compute the final charge of $3$-faces and $3$-vertices together. 

\begin{claim}
Every $4^+$-vertex $v$ has nonnegative final charge. 
\end{claim}
\begin{proof}
Note that $v$ is incident with at most $\lfloor{d(v)\over 2}\rfloor$ $3$-faces since there are no $4$-cycles. 
If $v$ is a $119^-$-vertex, then by Lemma~\ref{lem:unbal3:vx-degree}, $v$ has a neighbor $u$ that is a $120^+$-vertex. 
By (R2A), $u$ sends charge $2\over 3$ to $v$, and by (R3), $v$ sends charge at most ${1\over 3}\cdot\lfloor{d(v)\over2}\rfloor$ to its incident $3$-faces. 
Thus, $\mu^*(v)\geq d(v)-4+{2\over 3}-{1\over 3}\cdot\lfloor{d(v)\over2}\rfloor\geq 0$ when $d(v)\geq 4$. 

Now assume $v$ is a $120^+$-vertex. 
Then $v$ sends charge at most $2d(v)\over 3$ to its neighbors by (R2A) and $v$ sends charge at most ${3\over 5}\cdot{\lfloor{d(v)\over 2}\rfloor}$ to its incident $3$-faces by (R2B). 
Thus, $\mu^*(v)\geq d(v)-4-{2d(v)\over 3}-{3\over 5}\cdot{\lfloor{d(v)\over 2}\rfloor}\geq 0$ when $d(v)\geq 120$. 
\end{proof}

Note that a $6^+$-face and $5^+$-face sends charge at least $1\over 3$ and at least $1\over 5$, respectively, to each incident $3$-vertex.
In particular, a $5$-face that is incident with at least one and at least two $4^+$-vertices sends charge at least $1\over 4$ and at least $1\over 3$, respectively, to each incident $3$-vertex. 

\begin{claim}\label{claim:unbal3:3x}
Each $3$-vertex $v$ that is not incident with a $3$-face has positive final charge.
\end{claim}
\begin{proof}
By Lemma~\ref{lem:unbal3:vx-degree}, $v$ has a $120^+$-vertex $u$ as a neighbor. 
The faces incident with $v$ sends charge at least $3 \cdot \frac{1}{5}$ to $v$ by (R1) and $u$ sends charge $2\over 3$ to $v$ by (R2).
Also $v$ loses charge $1\over 15$ at most twice by (R4). 
Thus, $\mu^*(v)\geq -1+{3\over 5}+{2\over 3}-{2\over 15}>0$. 
\end{proof}

\begin{claim}\label{claim:unbal3:33}
If $f$ is a $3$-face that is incident with three $3$-vertices $x, y, z$, then the sum of the final charge of $f, x, y, z$ is nonnegative. 
\end{claim}
\begin{proof}
Let $x', y', z'$ be the neighbor of $x, y, z$, respectively, that is not on $f$. 
Since there are no $4$-cycles, $x',y',z'$ are pairwise distinct. 
By Lemma~\ref{lem:unbal3:vx-degree}, $x', y', z'$ are all $120^+$-vertices. 
Since $x,y,z$ are 3-vertices, $xx',yy',zz'$ are not contained in $3$-faces. 
Therefore each face that is adjacent to $f$ is incident with at least two $120^+$-vertices. 
Thus, each of $x, y, z$ receives charge at least $2\over 3$ from the incident faces by (R1).
Now $x', y', z'$ sends charge ${2\over 3}$ to $x, y, z$, respectively, by (R2A). 
Thus, $\mu^*(f)+\mu^*(x)+\mu^*(y)+\mu^*(z)\geq -4+3\cdot{2\over 3}+3\cdot{2\over 3}=0$.
\end{proof}

\begin{claim}\label{claim:unbal3:32}
If $f$ is a $3$-face $xyz$ that is incident with exactly two $3$-vertices $x$ and $y$, then the sum of the final charge of $f, x, y$ is nonnegative. 
\end{claim}
\begin{proof}
Let $x'$ and $y'$ be the neighbor of $x$ and $y$, respectively, that is not on $f$. 
Note that $x'$ and $y'$ are distinct, and $xx'$ and $yy'$ are not contained in any $3$-faces since $x$ and $y$ are 3-vertices and $G$ has no $4$-cycles. 

Assume $z$ is not a $120^+$-vertex.
This implies that $x'$ and $y'$ are both $120^+$-vertices by Lemma~\ref{lem:unbal3:vx-degree}.
Therefore each face that is adjacent to $f$ is incident with at least two $4^+$-vertices. 
Thus, each of $x$ and $y$ receives charge at least $2\over 3$ from the incident faces by (R1).
Now $x'$ and $y'$ sends charge ${2\over 3}$ to $x$ and $y$, respectively, by (R2A).
Also $z$ sends charge $1\over 3$ to $f$ by (R3). 
Thus, $\mu^*(f)+\mu^*(x)+\mu^*(y)\geq -3+2\cdot{2\over 3}+2\cdot{2\over 3}+{1\over 3}=0$.

Assume $z$ is a $120^+$-vertex.
This implies that $x, y, f$ receives charge $2\over 3$, $2\over 3$, at least $3\over 5$, from $z$ by (R2A), (R2A), (R2B), respectively; note that the sum of these charge is ${29\over 15}$. 
Let $f_{xy}, f_{zx}, f_{zy}$ be the face incident with $xy, zx, zy$, respectively, that is not $f$. 
It is possible that $f_{xy},f_{zy}$, and $f_{zx}$ are not pairwise distinct.
Assume that one of $x',y'$ is a $4^+$-vertex. 
Without loss of generality, we may assume that $x'$ is a $4^+$-vertex. 
By (R1), $f_{zx}$ and $f_{xy}$ gives charge at least ${1\over 3}$ and at least ${1\over 2}$ to $x$ and $x, y$, respectively. 
Also, $f_{zy}$ gives charge ${1\over 4}$ to $y$ by (R1). 
Thus, $\mu^*(f)+\mu^*(x)+\mu^*(y)\geq -3+{29\over 15}+{1\over 3}+{1\over 2}+{1\over 4}>0$.
So we may assume both $x'$ and $y'$ are $3$-vertices. 
By Lemma~\ref{lem:unbal3:vx-degree}, $x'$ and $y'$ must have a neighbor $x''$ and $y''$, respectively, that is a $120^+$-vertex. 
Note that $x''=y''$ is possible. 

If none of $x''$ and $y''$ is incident with $f_{xy}$, then $f_{xy}$ sends charge at least ${2\over 5}$ to $x$ and $y$ by (R1), and each $f_{zx}$ and $f_{zy}$ sends charge at least ${1\over 3}$ to $x$ and $y$, respectively, by (R1). 
Thus, $\mu^*(f)+\mu^*(x)+\mu^*(y)\geq -3+{29\over 15}+{2\over 5}+2\cdot{1\over 3}=0$.
If exactly one of $x''$ and $y''$ is incident with $f_{xy}$, then without loss of generality, we may assume $x''$ is incident with $f_{zx}$ and $y''$ is incident with $f_{xy}$. 
Now, by (R1), $f_{xy}$ sends charge at least $1\over 4$ to each of $x$ and $y$, and $f_{zx}$ and $f_{zy}$ sends charge at least $1\over 3$ and at least $1\over 4$ to $x$ and $y$, respectively. 
Thus, $\mu^*(f)+\mu^*(x)+\mu^*(y)\geq -3+{29\over 15}+{1\over 3}+3\cdot{1\over 4}>0$.

Assume both $x''$ and $y''$ are incident with $f_{xy}$.
If $x''\neq y''$, then $d(f_{xy}) \geq 6$ and $f_{xy}$ sends charge at least ${1\over 3}$ to each of $x$ and $y$ by (R1), and $f_{zx}$ and $f_{zy}$ sends charge at least ${1\over 4}$ to $x$ and $y$, respectively, by (R1). 
Thus, $\mu^*(f)+\mu^*(x)+\mu^*(y)\geq -3+{29\over 15}+{2\over 3}+2\cdot{1\over 4}>0$.
Now consider the case when $x''=y''$, so $f_{xy}$ sends charge $1\over 4$ to each of $x$ and $y$ by (R1). 
If one of $f_{zy}$ and $f_{zx}$ is either a $6^+$-face or a $5$-face that is not annoying, then it sends charge at least $1\over 3$ to $y$ or $x$ by (R1) and the other face still sends charge to $x$ or $y$ at least $1\over 4$ by (R1). 
Thus, $\mu^*(f)+\mu^*(x)+\mu^*(y)\geq -3+{29\over 15}+{1\over 3}+3\cdot{1\over 4}>0$.
So assume each of $f_{zx}$ and $f_{zy}$ is an annoying $5$-face, which sends charge $1\over 4$ by (R1). 
In particular, $x'$ and $y'$ have degree 3.
Therefore, $x$ and $y$ receive a total of charge $1$ by the surrounding faces. 

If $f'$ is a 3-face incident with $x'$, then it is incident with $x',x''$, and a vertex on $f_{zx}$ other than $z$.
Since $f_{zx}$ is an annoying 5-face, $f'$ is an annoying 3-face.
So every 3-face incident with $x'$ is annoying.
Similarly, every 3-face incident with $y'$ is annoying.
Since $xyy'x''x'$ is an annoying 5-face and $xyz$ is an annoying 3-face, either one of $x',y'$ is not incident with any 3-face, or some 3-face incident with both $x',y'$, by Lemma \ref{lem:unbal3:annoying}.
The later implies that $x'$ is adjacent to $y'$, which is a contradiction since $x'y'yx$ is now a $4$-cycle. 
Hence one of $x',y'$ is not incident with any 3-face, and that vertex sends charge $1\over 15$ to either $x$ or $y$ by (R4).
Thus, $\mu^*(f)+\mu^*(x)+\mu^*(y)\geq -3+{29\over 15}+1+{1\over 15}= 0$.
\end{proof}

\begin{claim}\label{claim:unbal3:31}
If $f$ is a $3$-face $xyz$ that is incident with exactly one $3$-vertex  $x$, then the sum of the final charge of $f$ and $x$ is nonnegative. 
\end{claim}
\begin{proof}
By Lemma~\ref{lem:unbal3:vx-degree}, $x$ has a neighbor $x'$ that is a $120^+$-vertex. 

Assume $x'\not\in\{y, z\}$. 
The sum of charge received from the faces incident with $x$ is at least $2 \cdot \frac{1}{3}$ by (R1).
Also, $x'$ sends charge $2\over 3$ to $x$ by (R2A). 
Each of $y$ and $z$ sends charge at least $1\over 3$ to $f$ by either (R2B) or (R3). 
Thus, $\mu^*(f)+\mu^*(x)\geq -2+4\cdot{1\over 3}+{2\over 3}= 0$.

So we may assume $x' \in \{y,z\}$.
Without loss of generality, assume $x'=y$. 
The sum of charge received from the faces incident with $x$ is at least $2 \cdot \frac{1}{4}$ by (R1).
Now $x'$ sends charge $2\over 3$ and $3\over 5$ to $x$ and $f$ by (R2A) and (R2B), respectively. 
Also, $z$ sends charge at least $1\over 3$ to $f$ by either (R2B) or (R3). 
Thus, $\mu^*(f)+\mu^*(x)\geq -2+2\cdot{1\over 4}+{2\over 3}+{3\over 5}+{1\over 3}>0$.
\end{proof}

\begin{claim}\label{claim:unbal3:30}
If $f$ is a $3$-face $xyz$ that is incident with no $3$-vertices, then the final charge of $f$ is nonnegative. 
\end{claim}
\begin{proof}
Since each of $x, y, z$ is a $4^+$-vertex, each of $x, y, z$ sends charge at least $1\over 3$ to $f$ by either (R2B) or (R3). 
Thus, $\mu^*(f)\geq -1+3\cdot{1\over 3}=0$.
\end{proof}

Since no $3$-vertex is contained in two different $3$-faces, the sum of the final charge on all $3$-faces and all $3$-vertices is nonnegative by Claims~\ref{claim:unbal3:3x}, ~\ref{claim:unbal3:33}, ~\ref{claim:unbal3:32}, ~\ref{claim:unbal3:31}, and~\ref{claim:unbal3:30}.


\section*{Acknowledgment}

We thank the referees for carefully reading the manuscript and helpful suggestions. 


%

%

\end{document}